\documentclass{amsart} 
\usepackage{amssymb,latexsym}
\usepackage[utf8]{inputenc} 
\usepackage[english]{babel}

\usepackage{amsthm} 
\usepackage{hyperref} 
\usepackage{xcolor}
\usepackage{todonotes}
\newtheorem{theorem}{Theorem}[section]

\newtheorem{proposition}[theorem]{Proposition}

\theoremstyle{definition} \newtheorem{definition}{Definition}[section]

\theoremstyle{remark}

\usepackage{amssymb}
\usepackage{amsmath}
 \usepackage{latexsym}
\usepackage{a4wide}
\usepackage{stmaryrd}
 \usepackage{amscd}
 \usepackage{shuffle}
 \usepackage{wasysym}
\usepackage[ruled,vlined]{algorithm2e}

\def \B+{\operatorname{B}}

\def \Id{\operatorname{Id}}

\def\ddto{\left.\frac{d}{dt}\right|_{t=0}}
\newcommand{\SO}{\operatorname{SO}}
\newcommand{\PSL}{\operatorname{PSL}}

\newcommand{\so}{{\mathfrak so}}
\renewcommand{\sl}{{\mathfrak sl}}
\newcommand{\RR}{\mathbb{R}}
\newcommand{\CC}{\mathbb{C}}

\newcommand{\g}{{\mathfrak g}}
\newcommand{\gl}{{\mathfrak gl}}
\newcommand{\h}{{\mathfrak h}}
\newcommand{\m}{{\mathfrak m}}
\newcommand{\Exp}{\operatorname{Exp}}
\newcommand{\diag}{\operatorname{diag}}

\newcommand{\y}{{\bf y}}
\newcommand{\z}{{\bf z}}

\newcommand\mik[1]{{\bf{#1}}}
\newcommand\miks[1]{\underline{#1}}
\newcommand\mikt[1]{{#1}_0}
\newcommand\mikp[1]{(\miks{#1};\mikt{#1})}
\def\Ad{\operatorname{Ad}}
\def\Aut{\operatorname{Aut}}
\def\dExp{\operatorname{dExp}}
\def\dExpinv{{\dExp}^{-1}}

\def \M{{\mathcal M}}
\def \A{{\mathcal A}}

\def\ad{{\operatorname{ad}}}

\newcommand\refp[1]{(\ref{#1})}
\newcommand\stset[2]{\left\{\left.#1\ \right|\ #2\right\}}
\newcommand{\XM}{{\mathcal X}\M}

\begin{document}


\thispagestyle{empty}

\title{Geometric Integration on Symmetric Spaces}

\author{Hans Munthe-Kaas}
\address{Lie St{\o}rmer Center, UiT The Arctic University of Norway \\and Department of Mathematics, University of Bergen.}
\email{hans.munthe-kaas@uit.no}

\date{\today}

\maketitle

\begin{abstract}
We consider geometric numerical integration algorithms for differential equations evolving on symmetric spaces. The integrators are constructed from canonical operations on the symmetric space, its Lie triple system (LTS), and the exponential from the LTS to the symmetric space. 
Examples of symmetric spaces are $n$-spheres and Grassmann manifolds, the space of positive definite symmetric matrices, Lie groups with a symmetric product, and elliptic and hyperbolic spaces with constant sectional curvatures. We illustrate the abstract algorithm with concrete examples. In particular for the $n$-sphere and the $n$-dimensional hyperbolic space the resulting algorithms are very simple and cost only ${\mathcal O}(n)$ operations per step. 
\end{abstract}


\section{Introduction} 
\emph{Symmetric spaces} are fundamental geometric objects in mathematics originating from the seminal work on non-Euclidean geometries by Gauss, Bolyai, Lobachevsky, Riemann, Klein and Lie in the 19'th century. 
Examples include:
\begin{itemize}
\item The $n$-sphere $S^n$, the surface of the unit ball in $\RR^{n+1}$. 
\item Hyperbolic spaces $H^n$ with constant negative sectional curvature.
\item Symmetric positive definite matrices with $A\cdot B = AB^{-1}A$. 
\item Lie groups $G$ with a symmetric product $g\cdot h = gh^{-1}g$ and subsets of $G$ closed under $\cdot$. 
\item Homogeneous spaces $\M = G/H$ where the Lie algebra of $G$ splits as $\g = \h\oplus \m$ such that $\h$ is the Lie algebra of $H$ and $\m$ is a Lie triple system. 
\item The real line $\RR$ with $x\cdot y = 2x - y$.
\item Grassman manifolds: the $p$-dimensional subspaces of $\RR^n$.
\end{itemize}
Symmetric spaces have also appeared in various numerical algorithms, such as e.g.\ matrix computations and splines~\cite{munthe2014symmetric,bogfjellmo2018numerical,camarinha1995splines}. 

In this paper we are concerned with numerical algorithms for integrating dynamical systems evolving on symmetric spaces. 
Over the last two decades extensive theories have been developed for numerical integration of Lie groups and homogeneous spaces, so-called \emph{Lie group integrators}~\cite{iserles2000lie}. 
A symmetric space can always be constructed as a homogeneous space, so Lie group integrators \emph{can} be used also for symmetric spaces. However,
this approach neglects the finer geometric structures of the symmetric space. In particular, the natural metric of a symmetric space is preserved by
a symmetric product, but generally not by a (left- or right) Lie group action.  Furthermore, the tangent space of the symmetric space is a triple system closed under double Lie brackets, but not under the single bracket of the Lie algebra. The exponential map from the triple system to the symmetric space is often much cheaper to compute than the corresponding exponential map from the Lie algebra to the Lie group.  

So, there are good reasons to consider \emph{intrinsic} integration algorithms on symmetric spaces. In this paper we propose integration schemes entirely built from operations which can be canonically obtained from the definition of a symmetric space  $(\M,\cdot)$ as a manifold  with a symmetric product. Algorithms are written out explicitly for several important examples. For the $n$-sphere and hyperbolic $n$-space, the computational cost of the algorithms are just ${\mathcal O}(n)$ per step, using Rodrigues type formulae. The concrete algorithms for these cases can easily be implemented without comprehending the theoretical framework of this paper. 

\section{Canonical integration on symmetric spaces}

\subsection{An outline of the algorithm} 
We discuss the basic idea of canonical symmetric space integration in a general geometric language, assuming some familiarity with symmetric spaces. Detailed mathematical definitions are given in Section~\ref{sec:definitions} and concrete examples in Section~\ref{sec:concrete}. 

The basic idea of the algorithm is very similar to the basic formulation of the RKMK algorithm on a Lie group $G$, as presented in~\cite{munthe1999high}. In Lie group case we seek the solution of a differential equation $y'(t) = f(y(t)) y(t)$ for $y(t)\in G$, where $f\colon G\rightarrow \g$ and $\g=T_eG$ is the Lie algebra. The equation is pulled back from $G$ to the linear space $\g$ by the ansatz $y(t) = \exp(\theta(t))$ for $\theta(t)\in \g$. A crucial point is that
 $\theta(t)$ is the solution of the 'dexpinv' equation $\theta'(t) = \operatorname{dexp}^{-1}_\theta f(\exp(\theta))$, where $\operatorname{dexp}$ denotes a trivialised tangent of the exponential mapping and $\operatorname{dexp}^{-1}_\theta$ can be computed by commutators on $\g$. The RKMK Lie group integrator is obtained by solving the dexpinv equation by a Runge--Kutta method, moving back and forth between $\g$ and $G$ using $\exp$. 

Unlike a Lie group, which has the identity $e$ as a special point, a symmetric space has no special point. We can choose any point $o\in \M$ as
a \emph{base point} for various constructions, as long as we ensure that the construction is geometrically independent of the chosen  point. It should, however, be remarked that from the perspective matrix representations and efficient numerical linear algebra, certain base points might be preferred. 
We call $(\M,\cdot,o)$ a \emph{pointed symmetric space}. The tangent space $\m:=T_o\M$ has the algebraic structure of a \emph{Lie Triple System},
$(\m,[\_,\_,\_])$, where $[\_,\_,\_]\colon \m\times\m\times\m\rightarrow \m$ is the \emph{triple bracket}, defined below. 

On $\M$ there exists a canonical connection $\nabla$ which is torsion free $T=0$ and has constant (parallel) curvature $\nabla R=0$. 
The algebra describing vector fields with the canonical connection on a symmetric space, called a Lie Admissible Triple algebra, is studied in~\cite{munthe2023lie}. 
The geodesic of this connection,
starting at a tangent $V\in \m$ is denoted $\Exp\colon \m\rightarrow \M$. The connection defines \emph{parallel transport} of tangent vectors along the geodesic. For a given $V\in \m$, parallel transport yields a linear isomorphism  $\Gamma_{V}\colon \m\rightarrow T_{\Exp(V)}\M$, which transports a tangent  vector along the geodesic of $V$ from $t=0$ to $t=1$. 

Crucial to our approach is the following characterisation of the tangent of the geodesic exponential $\Exp$ trivialised by  parallel transport with $\Gamma$:

\begin{proposition}\label{prop:dExp} For $V,W\in \m$ we have
\begin{equation}\label{eq:1}
\left.\frac{d}{dt}\right|_{t=0}\Exp(V+tW) = \Gamma_V \dExp_V W,
\end{equation}
where $\dExp_V\colon \m\rightarrow\m$ is given as
\begin{equation}
\dExp_V = \left.\frac{\sinh(\sqrt{x})}{\sqrt{x}}\right|_{x=\ad^2_V} =  \left.\sum_{n=0}^\infty\frac{x^n}{(2n+1)!}\right|_{x=\ad^2_V},
\end{equation}
for $\ad^2_V W := [W,V,V]$.
\end{proposition}

Proof: See~\cite{helgason2001differential}, Th. 4.1 p.215. A proof is also provided here, in Section~\ref{sec:Lie}. 

Consider the differential equation for $y(t)\in \M$ given a vector field $F\colon \M\rightarrow T\M$:
\begin{equation}\label{eq:diffeqn}
y'(t) = F(y(t)), \quad y(0) = o. 
\end{equation}
Assuming $y(t)\in U$ is within a normal neighbourhood  $o\in U\subset \M$, where $\Exp$ is well-defined, we write 
\begin{equation}y(t) = \Exp(\theta(t)), \quad \theta(t)\in \m, \end{equation} 
and
trivialise $F$ as
\begin{equation}
F(\Exp(\theta))  = \Gamma_{\theta} f(\Exp(\theta)), \quad \mbox{for $f\colon U\rightarrow \m$.}
\end{equation}

\begin{proposition}The curve $\theta(t)\in \m$ satisfies the differential equation
\begin{equation}
\theta'(t) = \dExpinv _\theta f(\Exp(\theta)), \quad \theta(0) = 0,
\end{equation}
for 
\begin{equation}
\dExpinv_{\theta} = \frac{\sqrt{x}}{\sinh(\sqrt{x})} =1-\sum_{n=1}^\infty (2^{2n}-2)\frac{B_{2n}}{(2n)!}x^{n}= 1 - \frac{x}{6}+ \frac{7x^2}{360}-\frac{31x^3}{15120} +{\mathcal O}(x^4),
\end{equation}
where $x = \ad^2_\theta = [\_,\theta,\theta]$ and $B_{2n}$ are the Bernoulli numbers. 
\end{proposition}

We obtain the first version of our algorithm by applying a Runge--Kutta (RK) method to $\theta(t)$:

\ \\
\begin{algorithm}[H]
\SetAlgoLined
\KwResult{Evolve $y'(t)=F(y(t))$ from $y_0 = o$ to $y_1 \approx y(t=h)$.}
Choose time step $h$ and $\{a_{i,j}\}_{i,j=1}^r$, $\{b_j\}_{j=1}^r$ coefficients of an RK method.\\
\For{$i=1,\ldots,r$}{
    $\theta_i =\sum_{j=1}^r a_{i,j} K_j$ \\
   $K_i = h\ \dExpinv_{\theta_i} \Gamma_{\theta_i}^{-1} F(\Exp(\theta_i))$ \\
  }
  $y_1 = \Exp(\sum_{j=1}^r b_{j} K_j)$\\
 \caption{Canonical integration step on symmetric space $(\M,\cdot,o)$, with $y_0=o$.}
 \label{alg:stepfromo}
\end{algorithm}\ \\

\emph{Remark.} The perhaps most important case is \emph{explicit RK methods}, where the coefficient matrix $\{a_{i,j}\}$ is strictly lower triangular. In this case $\theta_1=0$ and we find the others as
$\theta_i = \sum_{j=1}^{i-1}a_{i,j}K_j$. If the method is \emph{implicit}, a system of equations must be solved to find $K_i$ and $\theta_i$. 

\emph{Remark.} The $\dExpinv$ series for the symmetric space exponential starts with the double bracket, while for the Lie group exponential the series starts with a single bracket. In~\cite{munthe1995lie} we showed that a na\"{i}ve Lie group method, with no correction for $\dExpinv$ can generally only achieve order 2. The symmetric space algorithm implemented without a $\dExpinv$ correction can achieve order 3.
With higher order truncations of the $\dExpinv$ expansion, this method achieves the same order as the underlying RK method. 

This basic version of the method chooses the base point $o= y_0$. If we 
want to choose a different base point $o\neq y_0$, we can pull back the equation using an automorphism of $\M$ sending $o$ to $y_0$. 
The automorphisms $\Aut(\M)$ are all mappings $\tau\colon\M\rightarrow \M$ such that $\tau(x\cdot y) = \tau(x) \cdot \tau(y)$. 
Let $T\tau\colon T\M\rightarrow T\M$ denote the tangent map. 

\ \\
\begin{algorithm}[H]
\SetAlgoLined
\KwResult{Evolve $y'(t)=f(y(t))$ from $y_0$ to $y_1 \approx y(t=h)$.}
Choose any $\tau_{y_0}\in \Aut(\M)$ such that $\tau_{y_0}(o) = y_0$. \\
Compute $\tilde{y}_1$ by applying  Algorithm~\ref{alg:stepfromo} on the pull-back equation
\[\tilde{y}'(t) = T\tau_{y_o}^{-1} F(\tau(\tilde{y}(t))), \quad \tilde{y}(0) = o.\]
\flushleft Map back as $y_1= \tau_{y_0}(\tilde{y}_1)$.\\
 \caption{Canonical integration step on symmetric space $(\M,\cdot,o)$ with $y_0\ne o$.}
 \label{alg:symstep}
\end{algorithm}\ \\
We return to a more detailed description of these algorithms in Section~\ref{sec:definitions}, after a basic introduction to symmetric spaces.

\subsection{Defining the canonical operations} \label{sec:definitions}
We start with a pointed symmetric space $(\M,\cdot,o)$ and define all the operations of the algorithm from this structure. 
Concrete examples are found in Section~\ref{sec:concrete}. 
For a comprehensive treatment of the background theory, we refer to~\cite{loos1969symmetric,helgason2001differential}. 

\begin{definition}[Symmetric space]\label{def:symspace}
A symmetric space is a manifold $(\M,\cdot)$ with a product ${\cdot \colon\M\times \M\rightarrow \M}$ such that
\begin{align*}
x\cdot x & = x\\
x\cdot (x\cdot y)  &= y\\
x\cdot (y\cdot z) &= (x\cdot y)\cdot (x\cdot z)
\end{align*}
and every $x$ has a neighbourhood $U$ such that $x\cdot y=y$ implies $y=x$ for all $y\in U$.
\end{definition}

As an example, any Lie group $L$ with the symmetric product $x\cdot y:= xy^{-1}x$ is a symmetric space denoted $(L^+ ,\cdot)$. An other example
is the sphere $S^n = \{x\in \RR^{n+1}\colon x^Tx = 1\}$ with the product $x\cdot y = y - 2xx^T y$. 

For each point $x\in \M$ we have the symmetry $\sigma_x\colon \M\rightarrow \M$ defined as $\sigma_x(y) := x\cdot y$.
The symmetry $\sigma_x\in \Aut(\M)$ is an involutive automorphism of
the symmetric space $(\M,\cdot)$, having $x$ as an isolated fix point\footnote{To get a feeling for this structure, consider a Riemannian metric space with its geodesics. Moving from $y$ to $x$ along the joining geodesic twice the distance from $y$ to $x$, we arrive at $\sigma_x(y)$. So $\sigma_x$ is a reflection of the
geodesics through the point $x$. A Riemannian symmetric space has for each $x\in\M$ such a symmetry $\sigma_x$ which is a metric isometry and has $x$ as an isolated fix-point.}. 
Note: 'involutive': $\sigma_x(\sigma_x(y))=y$,  'automorphism': $\sigma_x(y\cdot z)=\sigma_x(y)\cdot \sigma_x(z)$,   isolated 'fix-point $x$':  $\sigma_x(x) = x$, 
are exactly the contents of Definition~\ref{def:symspace}

The product of two reflections is called a \emph{displacement}. 
The \emph{group of displacements} is generated by all $\sigma_x\sigma_y$ for $x,y\in \M$ and
denoted $G = G(\M)$. This is a normal subgroup of $\Aut(\M)$. 

Let $o\in \M$ be a chosen point called the \emph{base point}, and let $\m := T_o \M$. We call $(\M,\cdot,o)$ a \emph{pointed symmetric space}.
The map $Q\colon \M\rightarrow G(\M)$ defined as $x\mapsto Q_x := \sigma_x\sigma_o$, called the \emph{quadratic representation}. $Q_\M$ 
generates all of $G$ and it is a homomorphism of $(\M,\cdot,o)$ onto $(G^+,\cdot,e)$ with the
symmetric product $g\cdot h:= gh^{-1}g$, satisfying
\[Q_{x\cdot y} = Q_xQ_y^{-1}Q_x = Q_x\cdot Q_y, \quad Q_o = e\in G. \]
The quadratic representation defines an action of $\M$ on itself, $(x,y) \mapsto Q_xy\colon \M\times \M\rightarrow \M$. 
By differentiating $Q$ at $o$ we find the infinitesimal generator of the action as the vector field $\xi_V\in \XM$, defined for $V\in\m$ at $y\in\M$ as
\begin{equation}\label{eq:Exp1}
\xi_V(y)=    \left.\frac{d}{dt}\right|_{t=0} Q_{\gamma(t)}(y), \quad \gamma(0) = o, \quad \gamma'(0)=\frac12 V. 
\end{equation}
The factor $\frac12$ is due to the quadratic nature of $Q$. With this scaling we have $\xi_V(o)= V$, so it defines an extension of a tangent at $o$ to a vector field on all of $\M$. Also recall that $Q_y o$ sends $o$ beyond $y$ to the point on the opposite side of $y$ on the geodesic, so we have to travel with half the speed to arrive at $y$ for $t=1$. 

\begin{definition}[Geodesic $\Exp$] The geodesic exponential at the base point is a mapping $\Exp\colon \m\rightarrow \M$ defined as the $t=1$ solution of the differential equation
\begin{equation}
y'(t) = \xi_V(y(t)) , \quad y(0) = o. \label{eq:Exp2}
\end{equation}
\end{definition}

Let $(\RR,\cdot)$ be the real line with symmetric product $s\cdot t = 2s-t$. It can be shown that ${\Exp(s\cdot t V) = \Exp(sV)\cdot\Exp(tV)}$. 
We summarise the most important properties of the symmetric space exponential:

\begin{proposition} $\Exp\colon (\RR,\cdot,0)\rightarrow (\M,\cdot,o)$ is the unique pointed symmetric space homomorphism with derivative $V\in \m$ at $0$.
It relates to the quadratic representation via
\[\Exp(V) = Q_{\Exp(\frac12 V)}o.\]
\end{proposition}
Note that $Q_{\Exp(\frac12 tV)}o$ is a geodesic through $o$, while  $Q_{\Exp(\frac12 tV)}p$ is generally \emph{not} a geodesic through $p$ for $p\neq o$. 
The maps $Q_{\Exp(\frac12 tV)}\colon \M\rightarrow \M$ are $t$-parametrised families of automorphisms (isometries for Riemannian symmetric spaces) of $\M$. In the case of the sphere $S^2$, the family
$Q_{\Exp(\frac12 tV)}\colon S^2\rightarrow S^2$ rotates the sphere around an axis $V$ perpendicular to the north pole. Starting at the north pole, this yields a great circle (geodesic). Starting
at other points it yields (possibly) smaller circles orthogonal to the rotation axis $V$. 

For a fixed $y\in \M$, the tangent of the automorphism $p\mapsto Q_yp$ is denoted $TQ_y\colon T_p\M\rightarrow T_{Q_yp}\M$
\begin{equation}TQ_y(W) = \ddto TQ_y(\gamma(t)),\quad \gamma(0) = p, \gamma'(0)= W.\end{equation}
Restricted to $W\in \m$, this defines parallel transport along the geodesic:
\begin{definition}[Parallel transport of $W$ along the geodesic of $V$]\label{def:parallel} For $V,W\in \m$ define $\Gamma_V\colon \m\rightarrow T_{\Exp(V)} \M$ as
\[\Gamma_V W = TQ_{\Exp(\frac12 V)}(W). \]
\end{definition}

In the introduction we mentioned that symmetric spaces are equipped with  a canonical connection $\nabla\colon \XM\times\XM\rightarrow \XM$ for which  the name \emph{'geodesic curves'} acquires geometric meaning. For the Levi--Civita connection on a Riemannian symmetric space, the geodesics define (locally) shortest paths between points.
The canonical connection is torsion free and has constant curvature. We do not need $\nabla$ in this paper, so we will be brief on its construction. For  two vector fields $F,G\in \XM$  define the connection at the base point $o$ as the time derivative of the parallel transport,
\[\nabla_F G(o) = \left.\frac{d}{dt}\right|_{t=0} \Gamma_{tF(o)}^{-1}G(\Exp(tF(o))).\]
The definition can be extended to an arbitrary point $y\in \M$ by moving around with $\tau\in \Aut(\M)$ such that $\tau(o) = y$. We refer to~\cite{loos1969symmetric,kobayashi1963foundations} for details.

We define the last operation we need to integrate on $\M$,  the \emph{Lie triple system} on $\m$.
\begin{definition}[Triple bracket on $\m$]  The tri-linear bracket $[\_,\_,\_]\colon \m\times\m\times\m\rightarrow \m$  is defined as
\begin{equation}
[V,W,Z] = [[\xi_V,\xi_W]_J,\xi_z]_J (o),
\end{equation}
where $[\_,\_]_J$ is the Jacobi bracket of vector fields. 
\end{definition}

The algebra $(\m,[\_,\_,\_])$ has the following structure:

\begin{definition}[Lie triple system] A \emph{Lie triple system} (Lts) $(\A,[\cdot,\cdot,\cdot])$ is a vector space $\A$ with a tri-linear bracket $[\cdot,\cdot,\cdot]\colon \A\times\A\times\A\rightarrow \A$
satisfying for all $x,y,z,t,w\in \A$:
\begin{align}
&[x,x,z] = 0\label{eq:lts1}\\
&[x,y,z]+[y,z,x]+[z,x,y] = 0\label{eq:lts2}\\
&[x,y,[z,t,w]] = [[x,y,z],t,w] + [z,[x,y,t],w] + [z,t,[x,y,w]].\label{eq:lts3}
\end{align}
\end{definition}
We can combine Algorithms 1 and 2 in a single algorithm by using the quadratic representation to pull back $y_0$ to $\m$ and $\Exp$ to pull back from a neigbourhood of $o\in \M$ to $\m$. This yields:

\begin{proposition} Given $y'(t) = F(y(t))$, $y(0)=y_0$. Let $s\in \M$ be such that $Q_so = y_0$ and let
$\theta(t)\in\m$ such that $y(t) = Q_s(\Exp(\theta(t)))$. Then $\theta(t)$ satisfies
\[\theta'(t) = \dExpinv_\theta \Gamma_\theta^{-1}TQ_s^{-1} F(y(t)), \theta(0)=0.\]
\end{proposition}
Using a RK method on this equation yields the following algorithm:
\ \\
\begin{algorithm}[H]
\SetAlgoLined
\KwResult{Evolve $y'(t)=F(y(t))$ from $y(0)=y_0$ to $y_n \approx y(t\!=\!T)$.}
Choose time step $h=T/n$ and $\{a_{i,j}\}_{i,j=1}^r$, $\{b_j\}_{j=1}^r$ coefficients of an RK method.\\
\For{$\ell=0,\ldots,n-1$}{
  solve $Q_{s_\ell} o = y_\ell$ for $s_\ell$\\
\For{$i=1,\ldots,r$}{
  $\theta_i =\sum_{j=1}^r a_{i,j} K_j$ \\
  $K_i = h\ \dExpinv_{\theta_i} \Gamma_{\theta_i}^{-1} TQ_{s_\ell}^{-1}F\big(Q_{s_\ell}(\Exp(\theta_i))\big)$ \\
  }
  $y_{\ell+1} = Q_{s_\ell}\big(\Exp(\sum_{j=1}^r b_{j} K_j)\big)$\\
  }
 \caption{CSSI (Canonical Symmetric Space Integrator) $(\M,\cdot,o)$.}
 \label{alg:symmetric_integration}
\end{algorithm}\ \\

\section{Examples}\label{sec:concrete}
\subsection{Lie groups}\label{sec:Lie} Let $L$ be a Lie group with Lie algebra $(\g,[\_,\_]_{\g})$. For $V\in\g$ and $g\in G$ we write products in matrix style, e.g.\ $Vg$ instead of the more elaborate
$TR_g(V)$. 

Let $\exp\colon\g\rightarrow L$ denote the Lie group exponential. 
Any Lie group $L$ defines a pointed symmetric space denoted $L^+ := (L,\cdot,o)$, where $o=e$ is the identity of the Lie group and
$x\cdot y:= xy^{-1}x$. We obtain the quadratic representation $Q_x = \sigma_x\sigma_e$ and \[Q_xy = xyx.\] 
For any $V\in \g$ and $g\in L$,  we have $\xi_V(g) = \ddto Q_{\exp(\frac12 tV)}g = \frac12\big(Vg + gV\big)$. 
The solution of $y'(t) = \xi_V(y(t))$, $y(0) = g$ is $y(t) = \exp(\frac12 tV)g\exp(\frac12 tV)$, in particular for $g=e$ we have $y(t)=\exp(tV)$, so
we find  \[\Exp(V) = \exp(V).\] 
Let $V^r(g)=Vg$ and $V^\ell(g) = gV$. We have $[V^r,W^\ell]_J = 0$, $[V^r,W^r]_J = -[V,W]_\g^r$ and $[V^\ell,W^\ell]_J = [V,W]_\g^\ell$. Since $\xi_V = \frac12(V^r+V^\ell)$ we find the triple bracket
\[[V,W,Z] = [[\xi_V,\xi_W]_J,\xi_Z]_J(o) = \frac14 [[V,W]_\g,Z]_\g .\]
The parallel transport is \[\Gamma_V W = \ddto Q_{\Exp(\frac12 V)}\Exp(tW) = \exp(\frac12 V) W \exp(\frac12 V) .\]
We compute $\dExp$ and $\dExp^{-1}$ explicitly. From~\refp{eq:1} we have for $V,W\in \g$
\[\dExp_V W = \Gamma_V^{-1}\ddto \Exp(V+tW) = \exp(-\frac12 V)\ddto \exp(V+tW)\exp(-\frac12 V).\]
In the following computation, let the linear operators $x,y:\g\rightarrow \g$ be given as $xW=[V,W]_\g$ and $yW = [W,V,V] = \frac14[V,[V,W]_\g]_\g = \frac14x^2 W$. It is well known~\cite{iserles2000lie} that $\ddto \exp(V+tW) = \operatorname{dexp}_V W\exp(V)$ where $\operatorname{dexp}_V=(e^x-1)/x$. 
This yields, using $\Ad_{\exp(V)}= \exp(\ad_V)$: 
\begin{align*}
\dExp_V W &= \exp(-\frac12 V)\frac{e^x-1}{x}W\exp(V)\exp(-\frac12 V) = \Ad_{\exp(-\frac12 V)}\frac{e^x-1}{x}W =e^{-\frac{x}{2}}\frac{e^x-1}{x}W\\
&= \frac{\sinh(\frac{x}{2})}{x/2}W = \frac{\sinh(\sqrt{y})}{\sqrt{y}}W. 
\end{align*}
This establishes Proposition~\ref{prop:dExp}. It follows immediately that $\dExpinv_V = \frac{\sqrt{y}}{\sinh(\sqrt{y})}$.

This results in the following matrix version of Algorithm~\ref{alg:symmetric_integration}. As a specific application consider $y'(t) = F(y(t))$, where
$y(t)$ is a symmetric positive definite (SPD) matrix and $F(y)$ is symmetric. 
\ \\
\begin{algorithm}[H]
\SetAlgoLined
\KwResult{Evolve $y'(t)=F(y(t))$ from $y(0)=y_0$ to $y_n \approx y(t\!=\!T)$.}
Choose time step $h=T/n$ and $\{a_{i,j}\}_{i,j=1}^r$, $\{b_j\}_{j=1}^r$ coefficients of an RK method.\\
\For{$\ell=0,\ldots,n-1$}{
  solve $Q_{s_\ell} o = s_\ell s_\ell = y_\ell$ for $s_\ell$ (matrix square root)\\
\For{$i=1,\ldots,r$}{
  $\theta_i =\sum_{j=1}^r a_{i,j} \tilde{K}_j$ \\
  $K_i = h\  \Gamma_{\theta_i}^{-1} TQ_{s_\ell}^{-1}F\big(Q_{s_\ell}(\Exp(\theta_i))\big) = h\ \exp(-\frac{\theta_i}{2}) s_\ell^{-1} F\big(s_\ell\exp(\theta_i)s_\ell\big)s_\ell^{-1}\exp(-\frac{\theta_i}{2})$ \\
  $\tilde{K}_i = \dExpinv_{\theta_i} K_i = K_i - \frac{x}{6}K_i+ \frac{7x^2}{360}K_i-\frac{31x^3}{15120}K_i +\cdots$\\
 }
 $\theta_\ell = \sum_{j=1}^r b_{j} \tilde{K}_j$\\
  $y_{\ell+1} = Q_{s_\ell}\big(\Exp(\theta_\ell)\big) = s_\ell\exp(\theta_\ell) s_\ell$\\
  }
 \caption{CSGI (Canonical Symmetric Group Integrator)\\ $(\M,\cdot,o)$, $o=I$, $A\cdot B = AB^{-1}A$.}
 \label{alg:matrix_case}
\end{algorithm}\ \\
Here $xK_i = \frac14[[K_i,\theta_i],\theta_i]$ (matrix triple commutator), and the $\dExpinv$ series is truncated to the order of the RK method. 
The matrix square root is in general not uniquely defined. We return to the choice of square root in the end of Section~\ref{sec:lieaut}.

%

\subsubsection{Automorphisms on Lie groups and symmetric decompositions}\label{sec:lieaut} For general matrix Lie groups, it is not clear to us if the symmetric space integration algorithm above has any advantages over the conventional Lie group integrators. However, there are many important cases where $L^+$ contains a symmetric sub algebra. In these cases the conventional Lie group integrator destroys this subspace structure, while it is respected by the symmetric integration algorithm. 

Given a Lie group $L$ with Lie algebra $\g$ and the corresponding symmetric space $L^+$.  An involutive automorphism on $L$ is $S\colon L\rightarrow L$ such that $S(xy) = S(x)S(y)$ and $S^2=\Id$. There are two important subsets of $L$
\begin{align*}
L^S &= \{x\in L \colon S(x) = x\}\\
L_S & = \{x\in L \colon S(x) = x^{-1}\}. 
\end{align*}
The invariant elements $L^S$ is a subgroup of $L$ and the alternating elements $L_S$ is a symmetric subspace of $L^+$. 
Furthermore, $L_S$ is isomorphic to the homogeneous space $L/L^S$ and \emph{every} connected symmetric space $\M$ is isomorphic to a symmetric space of this kind. So this is really more than just an example, it is a generic case covering all connected symmetric spaces. 

On the algebra level the derivative of $S$ at the identity yields an involutive Lie algebra automorphism $dS\colon \g\rightarrow \g$, and the
algebra splits as $\g = \h\oplus \m$ where $\h$ is the $+1$ eigenspace of $dS$ and $m$ the $-1$ eigenspace. Here $\h$ is a Lie algebra and $\m$ a Lie triple system. The split parts induces a $Z_2$ grading on $\g$, where the sub spaces satisfy $[\h,\h]\subset \h$, $[\h,\m]\subset \m$ and $[\m,\m]\subset\h$. 

There are many structured problems in matrix computations which fit into this format. A famous example is the group of all invertible matrices $L = \operatorname{Gl}(n)$ 
with the automorphism $S(A) = A^{-T}$. Here $L^S$ is the group of orthogonal matrices and $L_S$ the symmetric space of symmetric non-singular matrices.
An important sub algebra is the symmetric space of symmetric positive definite (SPD) matrices. On the algebra level, $dS(V) = -V^T$ 
splitting the space of matrices as ${\mathfrak gl}(n) = \h\oplus \m$ where $\h$ are the skew symmetric and $\m$ the symmetric matrices. Other examples
are discussed in~\cite{munthe2001generalized,munthe2014symmetric}. 

We return to the matrix square root in Algorithm~\ref{alg:matrix_case}. The final step is 
\[y_{\ell+1} = s_\ell\exp(\frac{\theta_\ell}{2})\exp(\frac{\theta_\ell}{2})s_\ell = S(A_{\ell+1})^{-1} A_{\ell+1}, \]
where $A_{\ell+1} = \exp(\frac{\theta_\ell}{2})s_\ell$. The \emph{generalised polar decomposition}~\cite{munthe2001generalized} is
$A_{\ell+1} = QP$, where $S(Q)=Q$ and $S(P)=P^{-1}$. 
For the SPD example, this is the classical polar decomposition where $Q$ is orthogonal and $P$ SPD. From this we find
\[y_{\ell+1} = S(A_{\ell+1})^{-1} A_{\ell+1} = PP,\] 
and hence $s_{\ell+1} = P$, the polar part of $A_{\ell+1}$. This polar part matrix square root is unique in many cases, such as for SPD matrices.

We finally note that in a dynamic situation, where $s_{\ell+1}$ is not too far from $s_\ell$, iterative techniques for the polar decomposition could be considered~\cite{higham1990fast}. There are many such practical issues we will not pursue here, but we will address these in forthcoming papers.

\subsection{Spheres}\label{sec:sphere}
Let $(S^n,\cdot,o)$ be the pointed symmetric space  $S^n = \{x\in \RR^{n+1} \colon x^Tx = 1\}$ with 
\[x\cdot y:= 2xx^Ty - y \]
The base point could be the north pole $o=(0,0,\ldots,0,1)$ (or any other point). 
The automorphisms are the orthogonal group $\Aut(S^n) = O(n+1)$ generated by the symmetries 
\begin{equation}\sigma_x = 2{xx^T}-I\in O(n+1).\end{equation}
The group of displacements $G(S^n)=\SO(n+1)$ is the group of pure rotations, and it is generated by the quadratic representation 
\[Q_x = \sigma_x\sigma_o = (2xx^T-I)(2oo^T-I),\]
where $Q_x$ is the rotation in the $o-x$ plane through twice the angle from $o$ to $x$. Let $\m$ be the horizontal vectors at the north pole
(the equatorial plane)
\[\m = T_oS^n=\{v\in \RR^{n+1} \colon v^To = 0\}\simeq \RR^n,\]
where we identify with $\RR^n$ by deleting last $0$ in the vector.
The infinitesimal generator is
\[\xi_V(p) = \ddto Q_{o+tv} p = \left(vo^T - ov^T\right) p = \hat{v} p,\]
where the hat map $\hat{v}:=  \left(vo^T - ov^T\right) \colon \m\rightarrow \so(n+1)$ is the unique map such that $\hat{v}o = v$. From this we
compute the triple bracket and the double adjoint on $\m$: 
\begin{align*}[u,v,w] &= [[\hat{u},\hat{v}],\hat{w}]o = \big(vw^T-v^Tw I\big)u \\
\ad^2_v w &= [w,v,v] = \big(vv^T-v^TvI\big)w .
\end{align*}
For $x=\ad^2v$ we find $x^2=-v^Tv x= (i\varphi)^2x$ for $\varphi = ||v||$. Since $\frac{\sqrt{x}}{\sinh(x)}=1+\sum_{j=1}^\infty a_jx^j$, we find
\[\dExpinv_v = \frac{\sqrt{x}}{\sinh(x)} =I+\left(\frac{i\varphi}{\sinh(i\varphi)}-1\right)\frac{x}{-\varphi^2} = I+\left(\frac{\varphi}{\sin(\varphi)}-1\right)\pi_v^\perp,\]
where $\pi_v^\perp=\frac{x}{-\varphi^2} = I-\frac{vv^T}{v^Tv}$ is projection on the orthogonal complement of $v$. This formula is not hard to understand geometrically. In $\dExpinv_v w$ the normal component of $w$ is increased by the factor $\frac{\varphi}{\sin \varphi}$. This is exactly the ratio of the circumference of a circle in the plane $\m$ with radius $\varphi$ and the circumference of the constant latitude circle on the sphere at angle $\varphi$ from
the north pole.

The exponential $\Exp(v)$ can be found by solving ${y'(t) = \xi_v(y(t))}$, $y(0)=o$ at $t=1$, giving
$\Exp(v) = \exp(\hat{v})o$. 
We can compute the matrix exponential $\exp(\hat{v})$ via Rodrigues formula~\cite{iserles2000lie}. From $\hat{v}^3 = -\varphi^2 \hat{v}$ for $\varphi=(v^Tv)^{\frac12}$
we find
\[\exp(\hat{v}) = I + \frac{\sin(\varphi)}{\varphi}\hat{v}+\frac12\frac{\sin^2(\varphi/2)}{(\varphi/2)^2}\hat{v}^2 .\]
This gives
\[\Exp(v) =\sin(\varphi)\frac{v}{||v||}+\cos(\varphi)o,\]
a rotation of $o$ in the $v$-$o$ plane through the angle $\varphi$.

We obtain the following spherical integrator, where we at each step choose a new base point $o = y_\ell$. 
The algorithm should never take steps $\theta_i\geq \pi$, because of the singularity in $\dExpinv$. If this happens, the step size $h$ must be reduced. 
\ \\
\begin{algorithm}[H]
\SetAlgoLined
\KwResult{Evolve $y'(t)=F(y(t))$ from $y(0)=y_0$ to $y_n \approx y(t\!=\!T)$,\\ \quad where $||y(t)||=1$ and $F(y)^Ty = 0$.}
{\bf Initialisation:} Choose time step $h=T/n$ and $\{a_{i,j}\}_{i,j=1}^r$, $\{b_j\}_{j=1}^r$ coefficients of RK.\\
\For{$\ell=0,\ldots,n-1$}{
\For{$i=1,\ldots,r$}{
  $\theta_i =\sum_{j=1}^r a_{i,j} \tilde{K}_j$ \\
  $\varphi = ||\theta_i||$\\
  \If {$\varphi==0$}{
       $K_i= hF(y_\ell)$ } 
  \Else{
  $E_i = \Exp(\theta_i) = \frac{\sin\varphi}{\varphi} \theta_i + \cos(\varphi) y_\ell$\\
  $s = \Exp(\theta_i/2) = (E_i+y_\ell)/||E_i+y_\ell||$ \\
  $v = h\ F\big(E_i\big)$ \\
  $K_i = \Gamma_{\theta_i}^{-1} v = v-2ss^Tv$ \\
  $\tilde{K}_i = \dExpinv_{\theta_i}K_i = K_i+\left(\frac{\varphi}{\sin(\varphi)}-1\right)\left(K_i-\frac{\theta_i\theta_i^T}{\varphi^2}K_i\right)$\\}
  }
  $\theta = \sum_{j=1}^r b_{j} \tilde{K}_j$\\
   $\varphi = ||\theta||$\\
  $y_{\ell+1} = \Exp(\theta) = \frac{\sin\varphi}{\varphi} \theta + \cos(\varphi) y_\ell$\\
  }
 \caption{CSI (Canonical Spherical Integrator)}
 \label{alg:CSI}
\end{algorithm}\ \\
The algorithm is not time symmetric for a time symmetric dynamical system, even if the underlying RK method is 'self adjoint' (time symmetric). 
The reason for this is the choice of
$o = y_\ell$ at each step, which is not symmetric with respect to time reversal. An alternative is to choose the base point as the geodesic midpoint $o = (y_\ell + y_{\ell+1})/||y_\ell + y_{\ell+1}||$. This yields a time symmetric integrator for time symmetric problems and self adjoint RK methods, see~\cite{zanna2001adjoint} for such methods in the Lie group case. But, of course, such methods are necessarily implicit. 
See remark on  relations to the Spherical Midpoint Method~\cite{mclachlan2017minimal} in Section~\ref{sec:remark}.

\subsection{Hyperbolic spaces} 
$H^n$ is the unique simply connected $n$-dimensional Riemannian manifold of constant  sectional curvature $-1$. We will present this via the ``hyperboloid model'',  an isometric embedding of $H^n$ in $n+1$ dimensional Minkowski space~\cite{gallier2020differential}. 

Let $\stset{\mik{z} = \mikp{z}\in \RR^{n+1}}{\miks{z}\in \RR^n, \mikt{z}\in \RR}$ be standard coordinates on Minkowski space, where $\mikp{z}$ denotes a column vector with $\mikt{z}$ in last position. 
Define the indefinite Minkowski inner product 
\begin{equation}\langle \y,\z\rangle = \mik{y}^T J \mik{z} =\mikt{y}\mikt{z} -\miks{y}^T\miks{z}\end{equation}
where $J = \diag(-1,\ldots,-1,+1)$. 
 Special relativity is defined on spacetime $\RR^4$, where $\miks{z}\in\RR^3$ is space and $\mikt{z}\in\RR$ is time. 

The subset at unit distance from the origin splits in two connected hyperboloids
\[H^\pm = \stset{\mik{z}=\mikp{z}\in \RR^{n+1}}{\langle\z,\z\rangle=1}=\stset{\mikp{z}}{\mikt{z}=\pm\sqrt{\miks{z}^T\miks{z}+1} }. \] 
We  identify $H^n := H^+$.  Let the \emph{Lorentz group} $O(n,1)$ be the matrix group preserving the Minkowski inner product:
\begin{align*}O(n,1) := &\stset{A\in \operatorname{GL}(n+1,\RR)}{\langle A\mik{y},A\mik{z}\rangle = \langle \mik{y},\mik{z}\rangle \mbox{ for all $\mik{y},\mik{z}\in \RR^{n+1}$}} \\= &\stset{A\in \operatorname{GL}(n+1,\RR)}{JAJ = A^{-T}}.\end{align*}

Polar decompositions in Lie groups are derived from involutive automorphisms on Lie groups \cite{munthe2014symmetric}. The classical matrix polar decomposition $A = US$, where $U^TU=I$ is orthogonal\footnote{We usually write orthogonal matrices `$Q$', but here we need to distinguish from the quadratic representation.} and $S$ is symmetric positive definite (SPD), is obtained from the involutive automorphism on $\operatorname{GL}(n+1,\RR)$ given as $\alpha(A) = A^{-T}$  by requiring $\alpha(U) = U$ and $\alpha(S) = S^{-1}$. An other involutive automorphism important for the Lorentz group is $\beta(A) = JAJ$. The two automorphisms commute, so  $\beta\alpha(A) = JA^{-T}J$ is also involutive. 
Note that $O(n,1)=\stset{A\in\operatorname{GL}(n+1,\RR)}{\beta\alpha(A)=A}$. 

\begin{proposition} The polar decomposition of $A\in O(n,1)$ is given as $A = US_{\miks{s}}$ where 
\begin{equation}U = \left(\begin{array}{cc}\widetilde{U} & 0 \\ 0 & u\end{array}\right)\end{equation}
for an orthogonal $\widetilde{U} \in O(n)$, $u=\pm 1$,  and 
\begin{equation}S_{\miks{s}} = \left(\begin{array}{cc}\widetilde{S} & \miks{s} \\ \miks{s}^T & \mikt{s}\end{array}\right)\end{equation}
for some $\miks{s}\in\RR^n$, $\widetilde{S}= \sqrt{I_n+\miks{s}\miks{s}^T}$ (SPD square root) and $\mikt{s}=\sqrt{1+\miks{s}^T\miks{s}}$, thus $(\miks{s};\mikt{s})\in H^n$. 
\end{proposition}

\begin{proof}An alternative and detailed proof is given in~\cite{gallier2020differential}. Here we sketch a structural argument, omitting some minor details. 
We seek $A = US$ where $\alpha(U) = U$ and $\alpha(S) = S^{-1}$. We have
\[US = A = \beta\alpha(A) = \beta\alpha(U)\beta\alpha(S)= \beta(U)\beta(S^{-1}),\]
which yields $\beta(U)=U$ and $\beta(S)=S^{-1}$. From $\beta(U)=U$ follows
\[U = \left(\begin{array}{cc}\widetilde{U} & 0 \\ 0 & u\end{array}\right),\]
and $\alpha(U)=U$ implies $\widetilde{U}$ is orthogonal and $u^2=1$. Writing 
\[S = \left(\begin{array}{cc} \widetilde{S} & \miks{s} \\ \miks{s}^T & \mikt{s}\end{array}\right)\] we find from $\beta(S)=S^{-1}$ that 
\[S^{-1} = \left(\begin{array}{cc} \widetilde{S} & -\miks{s} \\ -\miks{s}^T & \mikt{s}\end{array}\right) .\] Multiplying $SS^{-1} = I_{n+1}$ yields $\widetilde{S}^2-\miks{s}\miks{s}^T = I_n$, $\mikt{s}^2-\miks{s}^T\miks{s} = 1$ and $\widetilde{S}\miks{s} = \mikt{s}\miks{s}$. 
\end{proof}
Let $O_0(n,1)$ denote the matrices in $O(n,1)$ where $u=+1$.  This is the subgroup mapping $H^+\mapsto H^+$ and $H^-\mapsto H^-$. 
The matrices where $u=-1$ swap the two components of $H^\pm$, but they do not form a subgroup. Note that 
\[\det(S^2) = \det(SJS^{-1}J) = \det(SS^{-1}) = \det(I) = 1 ,\]
and since $S$ is SPD, we have $\det(S)=1$. 
Hence $\det(A) = \det(U) =u\det(\widetilde{U}) = \pm 1$. 
The special Lorentz group \[\operatorname{SO}(n,1):= \stset{A\in O(n,1)}{\det(A)=1}\]
has two connected components for $u = \pm 1$. 
We denote $\operatorname{SO}_0(n,1)$ the connected component containing the identity. 
$\operatorname{SO}(n,1)$ consists of those $A$ where $u\det(\widetilde{U})=1$ and $\operatorname{SO}_0(n,1)$ those where $u=\det(\widetilde{U})=1$. 

We will present $H^n$ as a homogeneous space. Let $o= ({\miks{0}};1)\in H^n$ be the base point. Since $So = (\miks{s};\mikt{s})$, it follows that both $O_0(n,1)$ and $\operatorname{SO}_0(n,1)$ act transitively on $H^n$. The stabiliser subgroup at $o$ consists of $A=US$ such that $Ao = o$. Multiplied out this implies
$S=I_{n+1}$ and 
\[U = \left(\begin{array}{cc}\widetilde{U} & \miks{0} \\ \miks{0} & 1\end{array}\right), \quad  \widetilde{U}\in O(n) .\]
Thus \[H^n = O_0(n,1)/O(n) \simeq \operatorname{SO}_0(n,1)/\operatorname{SO}(n) .\]

We seek the symmetric product on $H^n$. Note that $\sigma_{o} := J\in O_0(n,1)$ is the isometric reflection around $o$. Reflection around a general point $\mik{s}=\mikp{s}\in H^n$ is obtained by 
moving the point to $o$, reflect with $J$ and move back, 
\[\sigma_{\mik{s}} = S_{\miks{s}} \sigma_o S_{\miks{s}}^{-1} = \left(\begin{array}{cc} C & \miks{s} \\ \miks{s}^T & c\end{array}\right)\left(\begin{array}{cc} -I & {\bf 0} \\ {\bf 0}^T & 1\end{array}\right)
\left(\begin{array}{cc} C & -\miks{s} \\ -\miks{s}^T & c\end{array}\right) =  \left(\begin{array}{cc} -I-2\miks{s}\miks{s}^T & 2\mikt{s}\miks{s} \\ -2\mikt{s}\miks{s}^T & 1+2\miks{s}^T\miks{s}\end{array}\right) 
 .
\]
From this we find the more `obvious' expression for the isometric reflection
\begin{equation}\label{eq:ssymm}
\sigma_{\mik{s}} = 2\mik{s}\langle\mik{s},\_\rangle - I .
\end{equation}
The quadratic representation is
\begin{equation}\label{eq:Hquad}Q_{\mik{s}} = \sigma_{\mik{s}}\sigma_{o} = S_{\miks{s}}JS_{\miks{s}}^{-1}J = S_{\miks{s}}^2 = S_{2\mikt{s}\miks{s}}= \left(\begin{array}{cc} I_n+2\miks{s}\miks{s}^T & 2\mikt{s}\miks{s} \\ 2\mikt{s}\miks{s}^T & 1+2\miks{s}^T\miks{s}\end{array}\right) .\end{equation}

The computation of geodesic exponential $\Exp$, triple bracket and $\dExp^{-1}$ is very similar to the spherical case, Section~\ref{sec:sphere}. 
We compute the geodesic exponential from~(\ref{eq:Exp1})-(\ref{eq:Exp2}). From the embedding  $H^n\subset\RR^{n+1}$ we see that $\m=T_{o}H^n = \RR^n \simeq\{\mik{v}=(v;0)\in\RR^{n+1}\}$. 
Let $\gamma(t) = o + \frac12 t\mik{v} + \mathcal{O}(t^2)$. At $\mik{y}\in H^n$ we have
\begin{equation}\label{eq:xi}\xi_{v}(\mik{y}) = \left.\frac{\partial}{\partial t}\right|_{t=0}Q_{\gamma{t}} \mik{y}= \left.\frac{\partial}{\partial t}\right|_{t=0}\left(\begin{array}{cc} I_n & tv \\ tv^T & 1\end{array}\right) \mik{y} =  \left(\begin{array}{cc} {\bf 0} & v \\ v^T & 0\end{array}\right) \mik{y}  = \hat{\mik{v}}\mik{y} ,\end{equation}
where the hat-map in this case is
$\hat{\mik{v}} = \mik{v}\langle o,\_\rangle-o\langle \mik{v},\_\rangle$. 

\emph{Remark}: The structure of this hat map can be explained from the involutive automorphisms $\alpha,\beta\colon GL(n+1,\RR)\rightarrow GL(n+1,\RR)$. Differentiating these we get 
involutive automorphisms $d\alpha,d\beta\colon \gl(n+1,\RR)\rightarrow \gl(n+1,\RR)$
\[d\alpha(A) = -A^T, \qquad d\beta(A) = JAJ.\]
Since $\alpha(S) = S^{-1}$ and $\beta(S)=S^{-1}$, we have $S = \exp(V)$ for $d\alpha(V)=-V$ and $d\beta(V) = -V$.  
This implies that $V$ is of the form $V = \hat{\mik{v}}$. Let $d\alpha^-$ and $d\beta^-$ denote the -1 eigenspaces of these involutive automorphisms on $\gl(n+1,\RR)$, then $\hat{\m} :=\gl(n+1,\RR)\cap d\alpha^-\cap d\beta^-$ is a Lie triple system (LTS) of symmetric matrices and $\hat{\ }\colon \m \rightarrow \hat{\m}$ is the unique LTS isomorphism such that $\mik{v} = \hat{\mik{v}} o$. 
A matrix of the form of $S$ is called a \emph{Lorentz boost} in special relativity, and is the exponential of a matrix in $\hat{\m}$. 

From the hat map isomorphism we
compute the triple bracket and the double adjoint on $\m$: 
\begin{align*}[u,v,w] &= [[\hat{u},\hat{v}],\hat{w}]o = \big(v^Tw I-vw^T\big)u \\
\ad^2_v w &= [w,v,v] = \big(v^TvI-vv^T\big)w .
\end{align*}
Note that this is the negative of the triple bracket in the spherical case, which is not strange since the triple bracket is the Riemannian curvature tensor. 

For $x=\ad^2v$ we find $x^2= \varphi^2x$ for $\varphi = \sqrt{-\langle\mik{v},\mik{v}\rangle}$. As in the spherical case, we find
\begin{equation}\label{eq:Hdexpinv}\dExpinv_{\mik{v}} = \frac{\sqrt{x}}{\sinh(x)} =I+\left(\frac{\varphi}{\sinh(\varphi)}-1\right)\frac{x}{\varphi^2} = I+\left(\frac{\varphi}{\sinh(\varphi)}-1\right)\pi_{\mik{v}}^\perp,\end{equation}
where 
\begin{equation}\label{eq:Hproj}\pi_\mik{v}^\perp= I -\frac{\langle\mik{v},\_\rangle}{\langle \mik{v},\mik{v}\rangle}\end{equation} 
is the Minkowski projection on the orthogonal complement of $\mik{v}$. 

From~\refp{eq:xi} we find the  the geodesic exponential $\Exp\colon \m\rightarrow H^n$
\begin{equation}\Exp(\mik{v}) = \exp\left(\hat{\mik{v}}\right) o  .\end{equation}
Similar to the spherical case, $\hat{\mik{v}}^3 = \varphi^2 \hat{\mik{v}}$ yields a Rodrigues type formula 
\begin{equation}\exp(\hat{\mik{v}}) = I + \frac{\sinh(\varphi)}{\varphi}\hat{\mik{v}}+\frac12\frac{\sinh^2(\varphi/2)}{(\varphi/2)^2}\hat{\mik{v}}^2 . \end{equation}
This gives the geodesic exponential 
\begin{equation}\label{eq:Hexp}\Exp(\mik{v}) =\frac{\sinh(\varphi)}{\varphi}\mik{v}+\cosh(\varphi)o ,
\end{equation}
for $\langle \mik{v},o\rangle = 0$. This is a hyperbolic rotation of the $\mik{v}$-$o$ plane through the hyperbolic angle $\varphi$.

Similar to CSI Algorithm~\ref{alg:CSI} we will express our hyperbolic integrator using at each step  the base point $o=\mik{y}_n$, where each step is performed as in Algorithm~\ref{alg:stepfromo}. 
At an arbitrary point $o=\mik{y}_n\in H^n$
the tangent space is
\[\m = T_oH^n = \stset{\mik{w}\in\RR^{n+1}}{\langle \mik{w},o\rangle =  0} .\]
The geodesic exponential is still given by~\refp{eq:Hexp}, with inverse differential~\refp{eq:Hdexpinv}-\refp{eq:Hproj}. Finally, for $\theta\in \m$, we need $\Gamma_{\theta}^{-1}$, parallel transport from $T_{\Exp(\theta)}H^n$ to $\m$ along the joining geodesic. Let $\mik{s} = \Exp(\theta/2)$ be the midpoint on the geodesic. From Definition~\ref{def:parallel} we find
\[\Gamma_{\theta}^{-1} = TQ_{\mik{s}}^{-1} = T\left(\sigma_{\mik{s}}\sigma_o\right)^{-1}= T\left(\sigma_{o}\sigma_{\mik{s}}\right) = -T\sigma_{\mik{s}}\]
since $T\sigma_o = -I$.  From~\refp{eq:ssymm} this yields
\begin{equation}
\Gamma_{\theta}^{-1} = I - 2\mik{s} \langle\mik{s},\_\rangle.\end{equation}

\ \\
\begin{algorithm}[H]
\SetAlgoLined
\KwResult{Evolve $\mik{y}'(t)=F(\mik{y}(t))$,  $\mik{y}(t)\in H^n$,  from $\mik{y}(0)=\mik{y}_0$ to $\mik{y}_n \approx \mik{y}(t\!=\!T)$,\\ \quad where $\langle \mik{y}(t),\mik{y}(t)\rangle =1$ and $\langle F(\mik{y}),\mik{y}\rangle = 0$.}
{\bf Initialisation:} Choose time step $h=T/n$ and $\{a_{i,j}\}_{i,j=1}^r$, $\{b_j\}_{j=1}^r$ coefficients of RK.\\
\For{$\ell=0,\ldots,n-1$}{
\For{$i=1,\ldots,r$}{
  $\theta_i =\sum_{j=1}^r a_{i,j} \tilde{K}_j \in \m = T_{\mik{y}_{\ell}}H^n$ \\
  $\varphi = \sqrt{-\langle\theta_i,\theta_i\rangle}$\\
  \If {$\varphi==0$}{
       $K_i= hF(\mik{y}_\ell)$ } 
  \Else{
  $\mik{u} = \Exp(\theta_i) = \frac{\sinh(\varphi)}{\varphi} \theta_i + \cosh(\varphi) \mik{y}_\ell$\\
    $\mik{s} = \Exp(\theta_i/2) = \frac{\sinh(\varphi/2)}{\varphi} \theta_i + \cosh(\varphi/2) \mik{y}_\ell$\\
 $\mik{v} = h\ F\big(\mik{u}\big)$ \\
  $K_i = \Gamma_{\theta_i}^{-1} \mik{v} = \mik{v}-2\mik{s}\langle\mik{s},\mik{v}\rangle$ \\
  $\tilde{K}_i = \dExpinv_{\theta_i}K_i = K_i+\left(\frac{\varphi}{\sinh(\varphi)}-1\right)\left(K_i-\frac{\langle\theta_i,K_i\rangle}{\varphi^2}\theta_i\right)$\\}
  }
  $\theta = \sum_{j=1}^r b_{j} \tilde{K}_j$\\
   $\varphi = \sqrt{-\langle\theta,\theta\rangle}$\\
  $\mik{y}_{\ell+1} = \Exp(\theta) = \frac{\sinh\varphi}{\varphi} \theta + \cosh(\varphi) \mik{y}_{\ell}$\\
  }
 \caption{CHI (Canonical Hyperbolic Integrator)}
 \label{alg:hyperbolic_integration}
\end{algorithm}\ \\
\subsubsection{Other models of $H^n$} There are other well known models of $H^n$, most famous are the Poincar\'{e} \emph{half space} and the \emph{disc models}. These are obtained from the  hyperboloid model by stereographic projections. The half space model is the stereographic projection from infinity along the edge of the light cone in the $x_n$ - $x_0$ plane, where the hyperbolic space is realised as the upper half space of $\RR^n$ given by $x_n>0$, or upper half plane of  $\CC$ in the case $n=2$. The disc model is the stereographic projection of $H^n$ from the point $(\miks{0};-1)$, realising the hyperbolic space as the unit ball in $\RR^n$, or unit disc $\CC$ in the case $n=2$. These models are beautiful, in particular for $n=2$ where the hyperbolic space becomes a homogeneous space under the action of the projective linear fractional transformations $\PSL(2,\RR)$ on $\CC$. However, these models are mathematically equivalent to the model we have presented above, so we omit the details. Spaces of constant negative curvature are important in the geometric theory of dynamical systems. E.g.\ Anosov flows on compact surfaces  $H^2/\Gamma$, where $\Gamma$ is a discrete subgroup of $\PSL(2,\RR)$  (Fuchsian group) are generic examples of  Axiom A type dynamical systems with chaotic solutions. Variants of the CHI algorithm should be interesting for numerical studies of Anosov flows. 
In numerical analysis, Riemannian spaces of negative curvature have recently appeared in geometric generalisations of classical stability theory to Riemannian manifolds~\cite{arnold2023b}. The present algorithms should be of interest in the development of these ideas. 

\section{Final remarks}\label{sec:remark}
We have in this paper developed the general concept of canonical geometric integration on symmetric spaces. The important  issue of numerical qualities of these algorithms and their applications to practical computational problems are left to sequel papers. Here we briefly remark on  interesting questions to be pursued, and  possible applications. 

Differential equations evolving on spheres are ubiquitous, with important examples from robotics, computational mechanics, rigid body dynamics and flows on planetary surfaces. Infinite dimensional spheres are natural domains for partial differential equations preserving energy or $L_2$ norms, such as the Schr\"{o}dinger and  Korteweg--de Vries equations. It is interesting that the CSI algorithm on the $n$-sphere is not significantly more expensive than classical RK methods on the embedding of the sphere in $\RR^{n+1}$. 

Hyperbolic geometry is the foundation of special relativity. Our CHI algorithm is formulated in terms of Lorentz transformations and should be interesting for geometric integration of Maxwell equations and  dynamical systems in special relativity, as it  preserves the Lorentz invariance of the equations. 

An other interesting class of equations evolving on symmetric spaces are Lie-Poisson systems. Hamiltonian mechanics are naturally formulated on cotangent bundles of Lie groups. By symmetry reduction, these can often be reduced to Lie-Poisson systems evolving on the dual of Lie algebras. The natural coadjoint action foliates the dual Lie algebras in symplectic leafs, which in many cases are symmetric spaces. For example spheres are symplectic leaves for $\so(n)^*$ and hyperbolic spaces for $\sl(n)^*$. 
Questions of symplectic algorithms for such problems are interesting, hard and largely open. We remark that the CSI integrator based on the implicit midpoint RK and the geodesic midpoint as base point for each step is \emph{not} the same as the celebrated symplectic spherical midpoint method~\cite{mclachlan2017minimal}. However the difference between these two methods is closely related to the difference between using the exponential and the Cayley transform as coordinate mappings. For use of Cayley maps,  we refer to the pioneering work on symplectic integration on spheres by Lewis and Simo~\cite{lewis1994conserving}.
Symplectic integration on hyperbolic symplectic leaves is to our knowledge largely unchartered territory. 

Differential equations evolving on the  space of symmetric positive definite matrices occur in different contexts. One example is in the inverse eigenvalue problem for SPD Toeplitz matrices, formulated as isospectral flows~\cite{calvo1997numerical}. A different example is the tracing of nerve fibres in diffusion tensor imaging of the brain, where the voxels are symmetric positive definite diffusion tensors. 

We have in this paper not written out the example of Grassman manifolds, which is yet another example of symmetric spaces. Differential equations evolving on Grassman manifolds has been considered by several authors, see~\cite{beck2018grassmannian} and references within. 

\bibliographystyle{plain} \bibliography{CSI}

\end{document}